\documentclass{article}
\usepackage{enumerate}
\usepackage{amssymb}
\usepackage{amsmath}
\usepackage{latexsym}
\usepackage{amsthm}
\newtheorem{teononum}{Theorem}
\newtheorem{theorem}{Theorem}[section]
\newtheorem{lemma}[theorem]{Lemma}
\newtheorem{proposition}{Proposition}[section]
\theoremstyle{definition}
\newtheorem{definition}[theorem]{Definition}

\theoremstyle{remark}
\newtheorem{remark}[theorem]{Remark}

\numberwithin{equation}{section}
\begin{document}

\title{Fields of algebraic numbers with bounded local degrees and their properties}
\author{Sara Checcoli\footnote{Dipartimento di Matematica, Universit\`{a} di Pisa, Largo Bruno Pontecorvo 5,
56127, Pisa, Italy. email:checcoli@mail.dm.unipi.it}}
\maketitle

\begin{abstract}
We provide a characterization of infinite algebraic Galois extensions of the rationals with uniformly bounded local degrees, giving a detailed proof of all the results announced in Checcoli and Zannier\rq{}s paper \cite{Che} and obtaining relevant generalizations for them. In particular
we show that that for an infinite Galois extension of the rationals the following three properties are equivalent: having uniformly bounded local degrees at every prime; having uniformly bounded local degrees at almost every prime; having Galois group of finite exponent. The proof of this result enlightens interesting connections with  Zelmanov's work on the Restricted Burnside Problem. We give a formula to explicitly compute bounds for the local degrees of an infinite extension in some special cases. 
We relate the uniform boundedness of the local degrees to other properties: being a subfield of $\mathbb{Q}^{(d)}$, which is defined as the compositum of all number fields of degree at most $d$ over $\mathbb{Q}$; being generated by elements of bounded degree. We prove that the above properties are equivalent for abelian extensions, but not in general; we provide counterexamples based on group-theoretical constructions with extraspecial groups and their modules, for which we give explicit realizations.
\end{abstract}

\maketitle

\section*{Introduction}
This work is a study of the relations between some properties which can occur for an infinite algebraic extension $K$ of the rationals.

The first property we shall consider is the uniformly boundedness of the local degrees of $K$, namely the existence of a constant $b$, depending only on the field $K$, such that for every prime number $p$ and every place $v_p$ of $K$ which extends the
$p$-adic one, the completion of $K$ with respect to $v_p$ is a finite extension of
$\mathbb{Q}_p$ of degree bounded by $b$.

The second property is the inclusion of the extension $K$ in a field $\mathbb{Q}^{(d)}$, which is defined as the compositum in the algebraic closure of the rationals $\mathbb{Q}^{\text{alg}}$ of all number fields of degree at most $d$ over $\mathbb{Q}$. 

The inquiry into the relation between these two properties was motivated by Bombieri and Zannier\rq{}s paper \cite{Zan}, which studies
the Northcott property (of the finiteness of elements of bounded absolute Weil height)
for certain infinite extensions of $\mathbb{Q}$. 
The paper considers in particular the field $\mathbb{Q}^{(d)}$ and the Northcott property is proved for the compositum of all
abelian extensions of $\mathbb{Q}$ of bounded degree, while for $\mathbb{Q}^{(d)}$
the question remains open. In this proof a crucial role is played by the uniform
boundedness of local degrees of the field $\mathbb{Q}^{(d)}$. 

The considerations made in \cite{Zan} led to the question, not free of independent interest, of
whether 
every infinite algebraic extension of $\mathbb{Q}$ with uniformly bounded local degrees 
is contained in $\mathbb{Q}^{(d)}$ for some positive integer $d$. 
This question was negatively answered in \cite{Che}, where the authors provide a counterexample built up from a certain family of $pq$-groups. In this work we ask, more generally, if for an algebraic extension of $\mathbb{Q}$, there are properties 
 which are equivalent to having uniformly bounded local degrees.

The study of the structure of $\mathbb{Q}^{(d)}$ revealed some unexpected sides. These arise when we consider a third property for an infinite algebraic extension $K$ of $\mathbb{Q}$, that is whether every finite subextension of $K$ can be generated by elements of degree bounded by a constant depending only on $K$.
A field with this property is certainly contained in $\mathbb{Q}^{(d)}$ for some positive integer $d$ and one could ask whether this condition is equivalent to being a subfield of $\mathbb{Q}^{(d)}$.
 
In this work we give a complete answer to all these questions proving the following result.
\begin{teononum}\label{main}
Let $K/\mathbb{Q}$ be an infinite Galois extension. Then the following conditions are equivalent:
\begin{enumerate}[(1)]
\item $K$ has uniformly bounded local degrees at every prime;
\item $K$ has uniformly bounded local degrees at almost every prime;
\item $\text{Gal}(K/\mathbb{Q})$ has finite exponent.
\end{enumerate}
Moreover, if $K/\mathbb{Q}$ is abelian, then the three properties:
\begin{enumerate}[(a)]
\item\label{A1} $K$ has uniformly bounded local degrees; 
\item\label{B1} $K$ is contained in $\mathbb{Q}^{(d)}$ for some positive integer $d$;
\item\label{C1} every finite subextension of $K$ can be generated by elements of bounded degree;
\end{enumerate}
are equivalent. 
However, in general, we have that (\ref{C1}) implies (\ref{B1}) which implies (\ref{A1}) and none of the inverse implications holds.
\end{teononum}
 
The work is structured in the following way.

Section \ref{chap_ab} focuses on the study of the structure of the field $\mathbb{Q}^{(d)}$ and of Galois extension with uniformly bounded local degrees, proving that these are exactly those whose Galois group has finite exponent. Moroever we prove that the existence of a uniform bound for the local degrees at almost every prime implies the uniform boundedness of the local degrees at every prime.
The first part of the proof was already done in \cite{Che} and it consists of an application of Chebotarev\rq{}s density theorem; the second part is based on a theorem of Shafarevich on the number of generators of the Galois group of a $p$-extension of $p$-adic fields and on Zelmanov\rq{}s result on the Restricted Burnside Problem.
As in \cite{Che}, the case of abelian extensions is also discussed, proving the first part of Theorem \ref{main}. Explicit bounds for the local degrees are provided in some special cases.

In Section \ref{extraspecial} we recall the main properties of extraspecial groups. These groups shall enable us to construct examples of fields entailing the validity of the statements of Theorem \ref{main}. We also describe some irreducible modules of extraspecial groups over finite fields which shall be used in our constructions.

In Section \ref{nonab} we prove the main theorem of this work on the non-equivalence of properties (\ref{A1}), (\ref{B1}) and (\ref{C1}). The proof is based on some group-theoretical constructions with extraspecial groups and their modules, followed by an application of Shafarevich's Theorem
about the realization of solvable groups of odd order as Galois groups.

The use of Shafarevich\rq{}s result can actually be avoided by realizing explicitly groups constructed in Section \ref{nonab} and this will be done in Section \ref{realiz_extra}.
These constructions are based on some elementary methods in inverse Galois theory and on a result by Serre.

\section{Extensions with uniformly bounded local degrees} \label{chap_ab}
In this work we shall mainly investigate the mutual relations of three properties of an infinite algebraic extension of $\mathbb{Q}$: having uniformly bounded local degrees; being contained in the compositum of a family number fields of bounded degree; being generated by elements of bounded degree. 
In particular we could ask whether these three properties are equivalent.

In this section we study some properties of the compositum of fields with bounded degrees, in order to give some motivations to our questions, and we provide a characterization of Galois extensions with uniformly bounded local degrees depending only on the exponent of their Galois group. Finally we give a positive answer to our question  in the case of abelian extensions.

\subsection{Galois extensions with finite exponent}\label{motiv}
We fix some notation.
We let $d$ be a positive integer and $F$ a number field. We denote by $F^{(d)}$ the field obtained by taking the compositum of all algebraic extensions of $F$ of degree at most $d$ (over $F$). 

If $K$ is an infinite algebraic extension of $\mathbb{Q}$ and $p$ is a prime number , we define a valuation $v_p$ of $K$ which extends the $p$-adic valuation over $\mathbb{Q}$ in the following way: we fix an embedding $\sigma:K\hookrightarrow \mathbb{C}_p$, where $\mathbb{C}_p$ is the $p$-adic completion of $\mathbb{Q}_p^{\text{alg}}$ and we define the valuation $v_p$ of $K$ as $v_p(x):=w_p(\sigma(x))$ for every $x\in K$.

The completion of $K$ with respect to $v_p$ is denoted by $K_{v_p}$ and it is an algebraic extension of $\mathbb{Q}_p$. This extension in general is not finite, but sometimes it is. We are interested in the case when this extension is finite and moreover its degree is bounded, for every prime $p$ and every valuation $v_p$ above $p$, by a constant independent from $p$. 
We have the following definition.

\begin{definition} An infinite algebraic extension $K$ of $\mathbb{Q}$ has \emph{uniformly bounded local degrees} if there exists a positive integer $B$ such that $[K_{v_p}:\mathbb{Q}_p]<B$ for every prime number $p$ and every valuation $v_p$ of $K$ extending the $p$-adic valuation of $\mathbb{Q}$. 
\end{definition} 
 One of the questions posed in this work is whether every algebraic extension of $\mathbb{Q}$ with uniformly bounded local degrees is contained in $\mathbb{Q}^{(d)}$ for some positive integer $d$.
In order to give a motivation to this we need to show that the field $\mathbb{Q}^{(d)}$ has uniformly bounded local degrees.
\begin{proposition}\label{boundFv}
Let $F$ be a number field of degree $n$ over $\mathbb{Q}$ and let $v$
be any valuation of $F$. Let $w$ be an extension of $v$ to $F^{(d)}$ and denote by $F^{(d)}_w$ and $F_v$ the completions of $F^{(d)}$ and $F$ with respect to $w$ and $v$ respectively.
Then the local degree $[F^{(d)}_w :F_v ]$ is bounded solely in terms of $n$ and $d$.
\end{proposition}
\begin{proof}
It is well known that there are only finitely many
extensions of degree at most $d$ of the $p$-adic field $F_v$ and that a bound for their
number exists depending only on $d$ and $n$ (a formula for this number is given by Krasner in
\cite{Kras}). Therefore, the compositum of all these extensions has a degree over
$F_v$ which is finite and depends only on $d$ and $n$. In particular, this
compositum contains the completion of $F^{(d)}$ with respect to any
valuation $w$ extending the $p$-adic one.
Therefore the field $F^{(d)}$ has uniformly bounded local degrees.
\end{proof}
Next theorem is a generalization of Remark 2 of \cite{Che} and it provides a characterization of Galois extensions with uniformly bounded local degrees only in terms of the exponent of their Galois group. Moreover it asserts that the existence of a uniform bound for the local degrees at almost every prime implies the uniform boundedness of the local degrees at every prime.
\begin{theorem}\label{chebotarev} Let $L$ be a number field and $K/L$ an infinite algebraic
Galois extension. Then the following conditions are equivalent:
\begin{enumerate}[(1)] 
\item \label{1} $K$ has uniformly bounded local degrees at every prime of $L$;
\item \label{2} $K$ has uniformly bounded local degrees at almost every prime of $L$;
\item \label{3} $\text{Gal}{(K/L)}$ has finite exponent.
\end{enumerate}
\end{theorem}
\begin{proof}
Obviously (\ref{1}) implies (\ref{2}).
Now we suppose that (\ref{2}) holds, that is $K$ has local degrees uniformly bounded by a constant $B$ at every prime except some primes belonging to a finite set $S$.

We fix a finite Galois extension $E$ of $L$ contained in $K$ and we take
$\sigma\in \text{Gal}(E/L)$. By Chebotarev's Density Theorem there exist
a prime $\wp$ of $L$ with $\wp\not\in S$, a prime $\beta$ of $E$ unramified above
$\wp$ and a conjugate $\tau$ of $\sigma$ that generates the decomposition group
$D(\beta|\wp)$ which is cyclic and isomorphic to
$\text{Gal}(E_{\beta}/L_{\wp})$, where $E_{\beta}$ and $L_{\wp}$
denote the completions of $E$ and $L$ with respect to $\beta$ and $\wp$
respectively.
By assumption, $|\text{Gal}(E_{\beta}/L_{\wp})|\leq B$, thus
$\sigma^{B!}=\tau^{B!}=id$ and 
$\text{exp}(\text{Gal}(E/L))\leq B!$.
Since $\text{Gal}(K/L)$ is the inverse limit of the family $\{\text{Gal}(E/L)\}_E$,
where $E$ varies among the finite Galois extension of $L$ contained
in $K$, we have $\text{exp}(G)\leq B!$ and (\ref{3}) holds.

Now we suppose that (\ref{3}) holds for $K$ and we set $\exp(\text{Gal}(K/L))=b$. We want to prove that the local degrees of $K$ are uniformly bounded at every prime. 

We can write $K$ as the compositum of a family of number fields $\{K_m\}_m$, where $K_m/L$ is any finite Galois extension of $L$ contained in $K$ and we set $G_m=\text{Gal}(K_m/L)$.

We let $n=[L:\mathbb{Q}]$ and we fix a prime number $p$. For every $m$ we denote by $L_v$ and $K_{m,v}$ the completions of $L$ and $K_m$, respectively, with respect to any valuation $v$ extending the $p$-adic one.
We  recall that $\text{Gal}(K_{m,v}/L_v)$, being isomorphic to a subgroup of $G_m$, has exponent at most $b$.
We have to discuss three cases for the extension $K_{m,v}/L_v$. 

If it is unramified, then it is cyclic of order bounded by $b$;
if it is tamely ramified, then it is a metacyclic extension, thus its order is bounded by $b^2$.

If it is wildly ramified, then the tamely ramified part has always degree at most $b^2$. The first ramification group  of $K_{m,v}/L_v$ is a $p$-group of exponent at most $b$. Moreover, by a theorem of Shafarevich (see \cite{Shaf}) on the number of generators of the Galois group of a $p$-extension of $p$-adic fields, this group has at most $n b^2+2$ generators.
So the problem reduces to the following: \emph{is it true that, for every positive integers $b$ and $m$, if a finite group $H$  has $m$ generators and exponent $b$, then the order of $H$ is bounded by a constant which depends only on $m$ and $b$?}

This question is known as the Restricted Burnside\rq{}s Problem and was positively answered in 1989 by Efim Zelmanov; for a detailed description of Zelmanov\rq{}s proof the interested reader should refer to Vaughan-Lee\rq{}s book, see \cite{Lee}.
In view of Zelmanov\rq{}s result, the wildly ramified part of the extension has a degree which is bounded by a constant depending only on $n$ and $b$.

We can now conclude our proof.
Summing up all the previous results, we get  that for every $m$, the local degree $[K_{m,v}:\mathbb{Q}_p]$ is bounded by a constant depending only on $b$ and $n=[L:\mathbb{Q}]$, which are fixed. Thus the number of all possible completions of the $K_m$\rq{}s at primes above $p$ is finite and independent from $p$. The compositum of this finite number of fields has degree over $\mathbb{Q}_p$ which is bounded by a constant independent of $p$ and it contains the completion of $K$ with respect to any valuation extending the $p$-adic one.\end{proof}
The proof of the previous result does not give an explicit bound for the local degrees of an extension of finite exponent. In order to compute effectively such a bound, we need additional informations on the length of an abelian series of the Galois groups of the local extensions, which are soluble groups. We recall the following definition.
\begin{definition}
A soluble group $G$ has \emph{derived length} $n$ if $n$ is the length of a shortest abelian series in $G$, i.e. a series $1=G_0\lhd G_1\lhd \ldots \lhd G_n=G$ in which each factor $G_{i+1}/G_i$ is abelian.
\end{definition}
We can now state the following result.
\begin{theorem}\label{bound_derived}
Let $K/\mathbb{Q}$ be an infinite Galois extension of exponent $b$ and let $\{K_m\}_m$ be the family of all finite Galois subextensions of $K$. Suppose that for every $m$, every prime number $p$ and every valuation $v$ of $K_m$ above $p$, the group $\text{Gal}(K_{m,v}/\mathbb{Q}_p)$ has derived length at most $n$. Then $K$ has local degrees bounded by $$\prod_{i=0}^{n} A(i)$$ where $A(0)=b^3$ and $$A(i+1)=b^{\left(\prod_{j=1}^i A(j)\right)+2}.$$\end{theorem}
\begin{proof}
We fix a prime number $p$ and we consider a valuation $w$ of $K$ extending the $p$-adic valuation of $\mathbb{Q}$. The field $K$ equals the compositum of all its finite Galois subextensions $\{K_m\}_m$ and, for every $m$, $v=w|_{K_m}$ is a valuation of $K_m$ above $p$. We denote by $K_w$ and by $K_{m,v}$ the completions of $K$ and $K_m$ with respect to $w$ and $v$, respectively.

The unramified part of the extension $K_{m,v}/\mathbb{Q}_p$ is contained in the compositum $L^{\text{ur},b}$ of all unramified extensions of $\mathbb{Q}_p$ of degree dividing $b$, which has degree $[L^ {\text{ur},b}:\mathbb{Q}_p]=b$, since we have only one unramified extension of $\mathbb{Q}_p$ of every fixed degree.
Therefore $K_w^{\text{ur}}$ is also contained in $L^{\text{ur},b}$.

The tamely ramified part of  the extension $K_{m,v}/\mathbb{Q}_p$
is contained in the maximal abelian extension $L^{\text{tame},b}$ of $L^{\text{ur},b}$ of exponent dividing $b$ and prime to $p$. From local class field theory we have $$\text{Gal}(L^{\text{tame},b}/L^{\text{ur},b})\simeq {L^{\text{ur},b}}^*/({L^{\text{ur},b}}^*)^a$$ for some integer $a$ prime to $p$ and smaller than $b$. Therefore $$ |\text{Gal}(L^{\text{tame},b}/L^{\text{ur},b})|\leq |(\mathbb{Z}/a\mathbb{Z})|^2\leq b^2.$$
Since this holds for every $m$, $K_w^{\text{tame}}$ is a subfield of $L^{\text{tame},b}$.

The wild ramified part of  the extension $K_{m,v}/\mathbb{Q}_p$, by Cauchy\rq{}s theorem, is nontrivial only when $p$ divides $b$ and it has
a Galois group which is  a $p$-group of exponent $p^{r_m}$ dividing $b$ and has derived length at most $n$. Therefore there exists a tower $$K_{m,v}^{\text{tame}}=M_0\subseteq M_1\subseteq \ldots \subseteq M_n=K_{m,v}$$ in which every step $M_{i+1}/M_i$ is an abelian $p$-extension.
We set $L_0=L^{\text{tame},b}$ and we denote by $L_{i+1}$  the compositum of all abelian $p$-extension of $L_i$ of exponent $p^{r_m}$.
We notice that $M_i$ is contained in $L_i$ for every index $i$. Again from local class field theory we have that \begin{equation}\label{formula}\text{Gal}(L_{i+1}/L_i)=({L_i}^*)/({L_i}^*)^{p^{r_m}}\simeq \left({\mathbb{Z}/p^{r_m}\mathbb{Z}}\right)^{[L_i:\mathbb{Q}_p]+\epsilon(L_i)}\end{equation}
where $\epsilon(L_i)=1$ if $L_i$ does not contain $p$-th roots of unity and $\epsilon(L_i)=2$ otherwise.
Since $p^{r_m}<b$ and $[L_0:\mathbb{Q}_p]\leq b^3$, we have $[L_1:L_0]\leq b^{b^3 +2}=A(1)$ and, using formula \ref{formula}, it is easy to see that $[L_i:L_{i-1}]\leq A(i)$.
Since $K_w$ is contained in $L_n$ the result follows.
\end{proof}
\subsection{Abelian extensions}\label{abelian}

We consider infinite abelian extensions of $\mathbb{Q}$, namely infinite algebraic Galois extensions of $\mathbb{Q}$ with abelian (profinite) Galois group. 
The following theorem, which was already in \cite{Che}, gives an answer to our initial question for these extensions.
\begin{theorem}\label{ab_case}
For an infinite abelian extension $K$ of $\mathbb{Q}$,  the following conditions are equivalent:
\begin{enumerate}[(a)]
\item\label{a_ab} $K$ has uniformly bounded local degrees;
\item \label{b_ab} there exists a positive
integer $d$ such that $K$ is contained in $\mathbb{Q}^{(d)}$;
\item\label{c_ab} every finite abelian subextension of $K$ can be generated by elements of bounded degree.
\end{enumerate}
\end{theorem} 
\begin{proof}
It is trivial to prove that (\ref{c_ab}) implies (\ref{b_ab}).
The fact that (\ref{b_ab}) implies (\ref{a_ab}) follows from Proposition \ref{boundFv}. 

We now assume that $K$ satisfies condition (\ref{a_ab}).
We set $G=\text{Gal}(K/\mathbb{Q})$; then, by Theorem \ref{chebotarev},
$\text{exp}(G)\leq b$ for some positive integer $b$ and
$$G=\varprojlim_m G_m$$ where
$G_m=\text{Gal}(K_m/\mathbb{Q})$ and $K_m$ is any finite
abelian extensions of $\mathbb{Q}$ contained in $K$.

For every $m$, $G_m$ is a finite abelian group and we can write it as a product of
finite cyclic groups
$$G_m=\prod_{i=1}^{n}U_i.$$
We let $H_i$ be the subgroup of $G_m$ defined as $$H_i:=\prod_{j\neq
i}U_j.$$ We have $[G_m:H_i]=|U_i|=\text{exp}(U_i)\leq b$ for all $i$'s and
$\cap_{i=1}^n H_i=1$.

Therefore $K_m$ is the compositum of the fields $\{K_m^{H_{i}}\}_i$ which satisfy the inequality
$[K_m^{H_{i}}:\mathbb{Q}]=[G_m:H_i]\leq b$. 
Thus $K_m$ is generated by elements of bounded degree and $K$ satisfies condition (\ref{c_ab}), which completes the proof.
\end{proof}

\begin{remark}
Since an abelian group is trivially soluble with derived length $1$, by Theorem \ref{bound_derived} we have that an abelian extension of exponent $b$ has local degrees bounded by $b^{b^3+5}$.
\end{remark}

\section{Extraspecial groups}\label{extraspecial}
In this section we recall definitions and properties of
extraspecial groups and their irreducible modules.
We do not make the whole theory of extraspecial groups, but we only recall some properties which shall be used to construct both
fields with  uniformly bounded local degrees not contained in
$\mathbb{Q}^{(d)}$ for any positive integer $d$ and subfields of $\mathbb{Q}^{(d)}$ which cannot be generated by elements of bounded degree.
A detailed discussion on extraspecial groups and related topics
can be found in Doerk and Hawkes\rq{} book \cite{Doerk}.

\subsection{Definition and properties}
\begin{definition}
Let $p$ be a prime. A $p$-group $G$ is said to be
\emph{extraspecial} if the center $Z(G)$ and the commutator subgroup
$G'$ have order $p$.
\end{definition}
We recall the main properties of extraspecial groups which shall be used in our constructions.
\begin{proposition}\label{extra_structure}
For every odd prime $p$ and every positive integer $m$ there exists an extraspecial group $G_m$ with the following properties:
\begin{enumerate} 
\item $G_m$ has order $p^{2m+1}$ and exponent $p$;
\item the quotient $G_m/Z(G_m)$ is elementary abelian of order $p^{2m}$;
\item $G_m$ is isomorphic to $(E_1\times\ldots\times E_m)/N_m$, where $E_1,\ldots,E_m$ are extraspecial groups of order $p^3$ and exponent $p$, the center of $E_i$ is generated by an element $z_i$ and $$N_m=\{(z_1^{a_1},\ldots,z_n^{a_n})|\sum_{i=1}^n a_i\equiv 0 \mod p\};$$
\item every abelian subgroup of $G_m$ has order smaller than $p^{m+1}$.
\end{enumerate}
\end{proposition}
\begin{proof}
See \cite{Doerk}, Chap. A, \S 19 and \S 20.
\end{proof}
\begin{remark}
It is easy to prove that every non abelian group of order $p^3$ is extraspecial and that, if $p$ is an odd prime, there exists a unique extraspecial group of order $p^3$ and exponent $p$, up to isomorphisms, namely the Heisenberg group with the presentation $$H=<x,y| x^p=y^p=1, [x,y]\in Z(H)>.$$
The structure theorem for extraspecial groups implies that every extraspecial group of order $p^{2m+1}$ and exponent $p$ satisfies all the properties of Prop. \ref{extra_structure} and it is unique up to isomorphisms (see \cite{Doerk}, Ch.A, \S 20, Thm. 20.5).
\end{remark}
\subsection{Some irreducible modules of an extraspecial group}\label{irr_module_extra}

We now want to give a description of certain irreducible modules of an extraspecial $p$-group over a finite field which shall be used to construct examples in the next Section. 

We fix two odd primes $p$ and $q$, with $p$ dividing $q-1$, such that the finite
field $\mathbb{F}_q$ contains a primitive $p$-th root of unity $\zeta$.
We consider the extraspecial group $H$ of order $p^3$, exponent $p$ and
generators $x$ and $y$ and we set $W=({\mathbb{F}_q})^p$.
Then $W$ has the structure of $H$-module via the following action: if
$\{e_1,\ldots,e_p\}$ is a basis for $W$, then $x$ permutes these elements
as a cycle of length $p$ and $y\cdot e_i=\zeta^i e_i$.

Now for every positive integer $m$ we denote by $G_m$ the extraspecial
group of order $p^{2m+1}$ and exponent $p$ satisfying Proposition \ref{extra_structure} and by $W_m$ the tensor product
of $m$ copies of $W$. $G_m$ is isomorphic to a quotient
of the direct product of $m$ copies of $H$, thus it acts on $W_m$, which
can be regarded as an $G_m$-module of dimension $p^m$ over $\mathbb{F}_q$.
We have the following proposition.
\begin{proposition}
$W_m$ is a faithful and absolutely irreducible $G_m$-module.
\end{proposition}
\begin{proof}
It is easy to see that the module $W_m$ is faithful for $G_m$. 
To prove that it is also absolutely irreducible we need to prove that that $W$ is an absolutely
irreducible $H$-module, since the tensor product of absolutely irreducible modules is absolutely
irreducible. This is equivalent to show that
$\text{End}_H(W)\simeq \mathbb{F}_q$.
Now a matrix $C$ belongs to $\text{End}_H(W)$ if and only if $C$ commutes
with the generators of $H$.
As before, we have that $H$ is generated by two elements $x$ and $y$ and
the matrices associated to this generators are the $p\times p$ matrices $X$ and $Y$ where:
\begin{enumerate}
 \item $X$ represents the cyclic permutation $(1,\ldots,p)$;
\item $Y$ is a diagonal matrix whose entries are the $p$ distinct
$p$-th roots of unity.
\end{enumerate}
We notice that $C\cdot Y=Y\cdot C$ if and only if $C$ is a diagonal
matrix, but since we must have $C\cdot X=X\cdot C$ then $C$ must be a
scalar matrix. Therefore $\text{End}_H(W)\simeq \mathbb{F}_q$.
\end{proof}

We end the section with a property on the intersection of all subgroups of an extraspecial group of bounded index; this result seems to be \emph{ad hoc} as it will be used only at the end of the next section to provide subfields of $\mathbb{Q}^{(d)}$ which cannot be generated by elements of bounded degree.
\begin{proposition}\label{luc}
Let $G$ be an extraspecial $p$-group of order $p^{2m+1}$. Then the
intersection of all subgroups of $G$ of index smaller then $p^{m}$ is
nontrivial.
\end{proposition}
\begin{proof} Let $H$ be a subgroup of $G$ such that $[G:H]<
p^{m}$, we have that $|H|> p^{m+1}$.
In view of Proposition \ref{extra_structure}, $H$ cannot be abelian and
it contains two elements which do not commute. This implies that
$G'\cap H\neq {1}$, hence $G'\subseteq H$, since the commutator
subgroup is cyclic of order $p$.
\end{proof}

\section{The main theorem}\label{nonab}

In Section \ref{chap_ab} we proved that a for an infinite abelian extension of $\mathbb{Q}$ there is a list of properties which are equivalent to the property of having uniformly bounded local degrees.
We now want to prove that this is not true in general. The non equivalence is entailed by the existence of fields constructed using extraspecial groups and their modules. 
We now state the main result which will be proved throughout this section and which provides a significant generalization of Theorem 1.1. of \cite{Che}.
\begin{teononum}\label{main_theorem}
Let $K$ be an infinite algebraic Galois extension of $\mathbb{Q}$. Consider the following properties for $K$:
\begin{enumerate}[(a)]
\item\label{a}  $K$ has uniformly bounded local degrees at almost every prime;
\item\label{b} $K$ is contained in $\mathbb{Q}^{(d)}$ for any positive integer $d$;
\item\label{c} every finite subextension of $K$ can be generated by elements of bounded degree over $\mathbb{Q}$;
\item\label{d} there exists a positive integer $b$ such that, whenever $K$ can be written as a compositum of a family $\{K_m\}_m$ of finite Galois extensions, then for every $m$ the Galois group of $K_m$ has only minimal normal subgroups of order at most $b$.
\end{enumerate}
Then:
\begin{enumerate}[(1)]
\item\label{c_impli_b_impli_a} $(\ref{c})\Rightarrow (\ref{b})\Rightarrow (\ref{a})$;
\item\label{b_impli_d} $(\ref{b})\Rightarrow (\ref{d})$.
\end{enumerate}
However no other implication holds, that is:
\begin{enumerate}[(i)]
\item\label{d_noimpli_b} $(\ref{d})\not\Rightarrow (\ref{b})$;
\item\label{b_noimpli_c} $(\ref{b})\not\Rightarrow (\ref{c})$;
\item\label{a_noimpli_b} $(\ref{a})\not\Rightarrow (\ref{b})$.
\end{enumerate}
\end{teononum}

\begin{remark}
 It is immediate to see that if $K$ satisfies (\ref{c}) then it also satisfies (\ref{b}) and, from Theorem \ref{boundFv}, it follows easily that $(\ref{c})\Rightarrow (\ref{b})\Rightarrow (\ref{a})$.
\end{remark}
As the proof of the result is quite long, we divide it into four different parts. 
\subsection{Proof of (2)}

We want to prove that (\ref{b}) implies (\ref{d}).
In order to do this we need the following lemma. 

\begin{lemma}\label{normal}
Let $G$ be a finite group with a minimal normal subgroup $W$
such that $|W|=m$. Suppose that $G$ is a quotient of a group $H$
where $H$ is a subgroup of a direct product $ H_1\times \ldots \times
H_s$.
Then $|H_i|\geq m$ for some index $i\in \{1,\ldots,s\}$
\end{lemma}
\begin{proof}
We suppose that $G=H/N$ with $H\leq H_1\times \ldots \times H_s$
and that $|H_i|<m$ for every $i$.
Let us denote by $\pi:H\rightarrow H/N$ the projection map.
We set $H_i:=H \cap 1\times \ldots \times 1\times H_i \times
\ldots \times H_r$ and we denote by $G_i=\phi(H_i)$ the image of
$H_i$ in $G$.  We notice that the $G_i$'s are all normal subgroups of
$G$.

We want to show by induction that $W\subseteq G_i$ for every $i$.

For $i=1$ this holds by assumption, since $G_1=G$.

Suppose it is true for $G_{i-1}$. Now $W\subseteq G_{i-1}$ and
$[G_{i-1}:G_i] <m$, thus $U:=G_i\cap W$ is normal and non trivial and so
it contains $W$. Therefore $W\subseteq G_i$. In particular $W\subseteq
G_s$, which is a contradiction, since $|G_s|\leq |H_s|<m$.
\end{proof}

Now we are able to prove our result. We suppose that $K\subseteq \mathbb{Q}^{(d)}$ for some integer $d$. We write $K$ as the compositum of all its finite Galois subextensions $\{K_m\}_m$ and we set $G_m=\text{Gal}(K_m/\mathbb{Q})$. We suppose moreover that $K$ does not satisfy property (\ref{d}).
This means that 
there exists a strictly increasing sequence of positive
integer $\{c_m\}_m$ such that $G_m$ has a minimal normal subgroup of order $c_m$. 

Thus, in view of Lemma \ref{normal}, whenever $G_m$
is isomorphic to a quotient $H/N$ and $H$ is a subgroup of a direct
product $H_1 \times \ldots \times H_s$, then $|H_i|\geq c_m$ for at
least one index $i$.

We fix $m$ such that $d!<c_m$. We notice that, since $K$ is contained in $\mathbb{Q}^{(d)}$, so is $K_m$; then there exist number
fields $L_1,\ldots,L_s$ such that $K_m\subseteq L_1\ldots L_s$ and
$[L_i:\mathbb{Q}]\leq d$.

We denote by $L_i^{\text{Gal}}$ the Galois closure of $L_i$ in
${\mathbb{Q}}^{\text{alg}}$ and by
$H_i=Gal(L_i^{\text{Gal}}/\mathbb{Q})$ its Galois group. We notice
that $|H_i|\leq d!<c_m$ by hypothesis.

We have $K_m\subseteq L_1^{\text{Gal}}\ldots L_s^{\text{Gal}}$ and 
$\text{Gal}(L_1^{\text{Gal}}\ldots L_s^{\text{Gal}})$ is isomorphic to a subgroup
of $H_1\times \ldots \times H_s$ via the restriction map.

Since $K_m$ is a Galois extension
$$G_m=\text{Gal}(K_m/\mathbb{Q})\simeq
\text{Gal}(L_1^{\text{Gal}}\ldots L_s^{\text{Gal}})/N$$ for some
normal subgroup $N$ of $\text{Gal}(L_1^{\text{Gal}}\ldots
L_s^{\text{Gal}})$.

Now $\text{Gal}(L_1^{\text{Gal}}\ldots L_s^{\text{Gal}})$ is a
subgroup of $ (H_1\times \ldots \times H_s)$ with $|H_i|\leq c_m$
for every $i$, which is a contradiction.
Thus $K$ satisfies (\ref{d}).

\subsection{Proof of (i)}

We have to prove that condition (\ref{d}) does not imply condition (\ref{b}). 
In order to do this we fix a prime $p$ and we consider a family $\{K_m\}_m$ of finite Galois extensions of  $\mathbb{Q}$  such that $G_m=\text{Gal}(K_m/\mathbb{Q})=\mathbb{Z}/p^m \mathbb{Z}$ is the cyclic group of order $p^m$.
We denote by $K$ the compositum of all the $K_m$\rq{}s.

We notice that $\exp(\text{Gal}(K/\mathbb{Q}))$ is not finite, since the group $\text{Gal}(K/\mathbb{Q})$ contains elements of order $p^m$ for every $m$. 
In view of Theorem \ref{chebotarev}, $K$ cannot have uniformly bounded local degrees and therefore it is not contained in $\mathbb{Q}^{(d)}$. 

However, for every $m$,  every minimal normal subgroup of $G_m$ has order $p$ and property (\ref{d}) holds.

\subsection{Proof of (ii)}

We want to show that condition (\ref{b}) does not imply condition (\ref{c}). 
We fix an odd prime number $p$ and we now consider an infinite family $\mathcal{F}=\{F_i\}_i$ of linearly disjoint Galois extensions of $\mathbb{Q}$, where, for every $i$,  $\text{Gal}(F_i/\mathbb{Q})=H$ is the extraspecial group of order $p^3$ and exponent $p$. 

We know that such a family exists: in fact a famous theorem of Shafarevich (see \cite{Shaf2}) asserts that every solvable group of odd order can be realized over $\mathbb{Q}$.
Now, for every integer $n$, the group $H^n=H\times \ldots \times H$, being a $p$-group with $p$ odd, is solvable and hence realizable. We denote by $F$ its realization, that is $F$ is a Galois extension of $\mathbb{Q}$ with $\text{Gal}(F/\mathbb{Q})=H^n$. We notice that $H^n$ has $n$ normal subgroups $H_1,\ldots, H_n$ isomorphic to $H^{n-1}$ whose intersection is trivial and $H^n/H_i=H$ for all $i$\rq{}s. Thus $F^{H_1}, \ldots, F^{H_n}$ are $n$ linearly disjoint realization of $H$ over $\mathbb{Q}$ and we can make this construction for every integer $n$. 

Now we consider the family of finite Galois extensions $\{K_m\}_m$ constructed in the following way: for every $m$, denote by  $L_m=F_1\cdot \ldots \cdot F_m$ the compositum of $m$ fields of the family $\mathcal{F}$ . By construction $$\text{Gal}(L_m/\mathbb{Q})=H\times \ldots \times H=H^m.$$
Now we consider the normal subgroup $N_m$ of $H^m$ described in Proposition \ref{extra_structure} such that $H^m/N_m$ is extraspecial of order $p^{2m+1}$ and we denote by $K_m:=L_m^{N_m}$ the subfield of $L_m$ fixed by this subgroup.
We have that $K_m/\mathbb{Q}$ is a finite Galois extension with Galois group an extraspecial group of order $p^{2m+1}$ and exponent $p$.

We denote by $K$ the compositum  of all the fields $\{K_m\}_m$.
Since by construction every field $K_m$ is contained in $\mathbb{Q}^{(p^3)}$, therefore $K\subseteq \mathbb{Q}^{(p^3)}$ and (\ref{b}) holds for $K$.

However, $K$ does not satisfy (\ref{c}). In fact suppose that every $K_m$ can be generated by elements of degree smaller than $d$, that is $K_m$ equals the compositum of all its subextensions of degree at most $d$ over $\mathbb{Q}$.

We fix $m$ such that $d<p^m$ and we denote by $\{H_{m,i}\}_i$
the family of all subgroups of $G_m$ of index smaller than
$c_m:=p^{m}$ in $G_m$. Then, by the Galois correspondence,
$\{K_m^{H_{m,i}}\}_i$ is the family of all subextensions of $K_m$ of
degree smaller than $c_m$ over $\mathbb{Q}$.
We notice that the compositum $F_m:=\prod_{i}K_m^{H_{m,i}}$ is
strictly contained in $K_m$ since, by Proposition \ref{luc}, we have
$$Z(G_m)\subseteq \bigcap_{i}H_{m,i}=Gal\left(K_m/F_m\right)\neq 1.$$
This contradicts the hypothesis on $K_m$.
\begin{remark}\label{construction_extra}
To avoid the use of Shafarevich\rq{}s result,
we give, in Section \ref{realiz_extra}, an explicit realization of the family $\mathcal{F}$ and of the fields $K_m$\rq{}s.
\end{remark}
\subsection{Proof of (iii)}

We now want to prove that (\ref{a}) does not imply (\ref{b}), that is we have to show the existence of an extension $K/\mathbb{Q}$
with uniformly bounded  local degrees such that, $K$ is not contained in $\mathbb{Q}^{(d)}$ for any integer $d$.

In view of Theorem \ref{chebotarev} and property (\ref{b_impli_d}), in  order to construct such a field, we have to find a family of finite Galois extensions $\{K_m\}_m$ such that, setting $G_m:=\text{Gal}(K_m/\mathbb{Q})$, the family of groups $\{G_m\}_m$ satisfies the following conditions: 
\begin{enumerate}\label{2prop}
\item\label{cond1} $\text{exp}(G_m)$ is bounded by a constant which does
not depend on $m$;
\item\label{cond2} there exists a strictly increasing sequence of positive
integer $\{c_m\}_m$ such that $G_m$
has a minimal normal subgroup of order $c_m$.
\end{enumerate}
Then the compositum $K$ of the fields $K_m$\rq{}s has the desired property.

In fact condition (\ref{cond1}) implies that $\exp(\text{Gal}(K/\mathbb{Q}))$ is finite, which is equivalent to the fact that $K$ has uniformly bounded local degrees. As for condition (\ref{cond2}), it implies that (\ref{b}) does not hold for $K$.
The construction of $K$ goes as follows.

We take $p$ and $q$ two odd primes such that $p$ divides $q-1$. 
We denote by $E_m$ the extraspecial group of order $p^{2m+1}$ and exponent $p$ satisfying Proposition \ref{extra_structure} and by $\mathbb{F}_q(E_m)$ its regular representation, that is
$$\mathbb{F}_q(E_m)=\{\sum_{g\in E_m}a_g g|a_g\in \mathbb{F}_q\}.$$
Now we consider the family of finite groups
$\{G_m\}_{m\geq 1}$ where $G_m=\mathbb{F}_q(E_m)\rtimes E_m$
and the semidirect product is taken via the natural action by translation of
$E_m$ on $\mathbb{F}_q(E_m)$.
We can now prove the following lemma.
\begin{lemma}\label{prop3}
The family $\{G_m\}_{m\geq 1}$ satisfies condition (\ref{cond1}) with bound $pq$ and condition (\ref{cond2}) with
$c_m= q^{p^{m}}$.
\end{lemma}

\begin{proof}
Since $\exp(G_m)=pq$ for every $m$, condition (\ref{cond1}) holds.

As for condition (\ref{cond2}), we consider the absolutely irreducible $E_m$-module $W_m$ constructed in Section \ref{irr_module_extra}.  Since $q$ and $p$ are different primes, by Maschke\rq{}s theorem (see \cite{Doerk}, Ch.A, \S 11, Thm. 11.5), the group algebra $\mathbb{F}_q(E_m)$ is a semisimple algebra. In particular the irreducible $E_m$-module $W_m$ appears as a direct summand of $\mathbb{F}_q(E_m)$. 
By abuse of notation, we denote by $W_m$ a subgroup of $\mathbb{F}_q(E_m)$ isomorphic to $W_m$ and  by $H_m$ the subgroup of $G_m$ given by $H_m:=W_m\rtimes 1$.

Then
the subgroup $H_m$ is a normal
subgroup of $G_m$ and it is also minimal, being isomorphic to $W_m$, which is an
irreducible module. Moreover it has order $|H_m|=q^{p^m}$ and therefore condition (\ref{cond2}) is true with $c_m=|H_m|=q^{p^m}$.
\end{proof}

Now, having constructed a family of groups which satisfy the
desired properties, we want to find a family of number fields realizing these groups. 

It is easy to show that such a family exists: in fact, the groups $G_m$\rq{}s are $pq$-groups and, by Burnside's Theorem
(see \cite{Doerk}, Ch. I, \S 2, p.210), they are
solvable. Therefore, in view of Shafarevich's Theorem (see \cite{Shaf2}), they can be realized over $\mathbb{Q}$ (recall that  $p$ and $q$ are odd primes), that is
there exists a family of number fields $\{K_m\}_{m\geq 1}$ such
that, for every $m$, $\text{Gal}(K_m/\mathbb{Q})\simeq G_m$ and we
denote by $K$ the compositum of this family. Then $K$ satisfies (\ref{a}), but not (\ref{b}) of Thereom \ref{main_theorem}.
\begin{remark}\label{construction_group_algebra}
Shafarevich\rq{}s result is not necessary to prove the existence of such a family, since this will be constructed explicitly in Section \ref{realiz_group_algebra}.
\end{remark}
Even if Theorem \ref{chebotarev} is sufficient to prove that the constructed field $K$ has uniformly bounded local degrees, one may be interested in knowing an effective bound for the local degrees of $K$.

In order to do this, we fix a prime number $\ell$. Since every extraspecial $p$-group $E_m$ has derived length $2$ and the group algebra $\mathbb{F}_q[E_m]$ is abelian, then the completion of every $K_m$ at a prime above $\ell$ is a Galois extension with Galois group of derived length at most $4$: in fact the Galos group of the tame part of the extension is metabelian and the first ramification group is abelian for $\ell\neq p$ and it is metabelian for $\ell=p$, being a subgroup of $E_m$. 

Since $p<q$, applying Theorem \ref{bound_derived} we obtain that an explicit bound for the local degrees of $K$ is given by $$q^{2(q^{2 q^6+10}+q^6+7)}.$$ 

However this bound can be significantly improved via the following considerations. 
We fix again a prime $\ell$ of $\mathbb{Q}$ and we denote by $v$ a valuation of $K_m$ above $\ell$. We have to discuss three cases.
\begin{enumerate} 
\item If $\ell\neq p,q$ then, for every $m$, the extension $K_{m,v}/\mathbb{Q}_\ell$ is tame and it is contained in the compositum of all metabelian extensions of $\mathbb{Q}_\ell$ of exponent dividing $pq$, which has degree bounded by $q^6$.

Then the local degrees of $K$ at prime $\ell\neq p, q$ are bounded by $q^6$.
\item If $\ell=q$ the unramified part of the extension is contained in the compositum $L^{\text{ur}}$ of all unramified extensions of $\mathbb{Q}_q$ of degree dividing $pq$ and this extension has degree bounded by $pq$. 

The tame and totally ramified part of the extension, is either trivial is cyclic of degree $p$. Therefore it is contained in the compositum $L^{\text{tame}}$ of all cyclic extensions of $L^{\text{ur}}$ of degree $p$, which has degree  bounded again by $p^2$.  

The wild ramified part, having Galois group which is a subgroup of $\mathbb{F}_q[E_m]$ and therefore elementary abelian, is contained in the compositum of all $q$-elementary abelian extensions of $L^{\text{tame}}$, which has degree over $L^{\text{tame}}$ at most $q^{p^{3}q+2}$. 

Therefore the local degree of $K$ at $q$ is at most $q^{p^{3}q+3}p^3<q^{q^4 +6}$. 
\item If $\ell=p$ we denote by $H_m$ the Galois group of the local extension $K_{m,v}/\mathbb{Q}_p$. Since $\text{Gal}(K_m/\mathbb{Q})=\mathbb{F}_q[E_m]\rtimes E_m$ is a $pq$-group with a unique $q$-Sylow subgroup $\mathbb{F}_q[E_m]$, which is normal, then $H_m$ also possesses a unique $q$-Sylow subgroup $Q_m$. 

By Schur-Zassenhaus\rq{}s theorem, since $p$ and $q$ are different primes, $H_m=Q_m\rtimes P_m$, where $P_m$ is the $p$-Sylow subgroup, $Q_m\leq \mathbb{F}_q[E_m]$, $P_m\leq E_m$ and the action of $P_m$ on $Q_m$ is induced by the action of $E_m$ on $\mathbb{F}_q[E_m]$ by restriction. 

Then the extension $K_{m,v}^{Q_m}/\mathbb{Q}_p$, having metabelian Galois group $P_m$, is contained in the compositum $L$ of all metabelian $p$-extensions of $\mathbb{Q}_p$, which has degree $[L:\mathbb{Q}_p]\leq p^{p^2+4}\leq q^{q^2+4}$. 

The extension $K_{m,v}/K_{m,v}^{Q_m}$ is an abelian $q$-extension of a $p$-adic field with $(p,q)=1$, thus it is contained in the compositum $L\rq{}$ of all tamely ramified extensions of $L$ of degree $q$ and we have $[L\rq{}:L]\leq q^2$.

Then the local degree of $K$ at $p$ is bounded by $q^{q^2+6}$.
\end{enumerate}
Summing up these results, $K$ has local degrees bounded by $q^{q^4+6}$.

\section{Explicit constructions}\label{realiz_extra}
In Section \ref{nonab} we saw that the existence of certain Galois extensions with special properties was among the basic ingredients of the proof of our main theorem; this was done by using Shafarevich\rq{}s Theorem on the realizability of solvable groups of odd order over $\mathbb{Q}$.

This section provides explicit realizations of those groups. 
These constructions are interesting for two reasons: firstly they permit to avoid the use of Shafarevich\rq{}s result. 
Secondly, coming back to one of the original motivations of this work, Widmer proves in \cite{Widmer} that the Northcott property for an infinite algebraic extension $K$ of $\mathbb{Q}$ is strictly related to the behavior of the discriminants of certain finite subextensions of $K$. Therefore, knowing how to concretely construct  fields described in Section \ref{nonab} could be a step towards understanding  whether these extensions might have the Northcott property or not.

Section \ref{realiz_extraspecial} is devoted to describe a method, based on results in Michailov\rq{}s paper \cite{Micha}, to realize extraspecial $p$-groups over a field of characteristic zero containing a primitive $p$-th root of unity. This entails the existence of the family of number fields used in the proof of  (\ref{b_noimpli_c}) of Theorem \ref{main_theorem}, as pointed out in Remark \ref{construction_extra}. In Section \ref{realiz_group_algebra} we use a method described by Serre in \cite{Serre} to solve embedding problems with abelian kernel. This result, together with the construction of Section \ref{realiz_extraspecial}, provides a realization of the family of Galois extensions needed to prove part (\ref{a_noimpli_b}) of the main theorem, as announced in Remark \ref{construction_group_algebra}. 

\subsection{Realizing extraspecial $p$-groups of exponent $p$}\label{realiz_extraspecial}

In this section we generalize a method described in \cite{Micha} which can be used to realize extraspecial $p$-groups of exponent $p$ over fields of characteristic zero containing a primitive $p$-th root of unity. We have the following theorem.
\begin{theorem}\label{teo_realize_extra}
Let $k$ be a field of characteristic zero containing a primitive $p$-th root of unity $\zeta_p$. Let $a_1,b_1,\ldots,a_m,b_m\in k$  be elements defined in the following way.
\begin{enumerate}[(i)]
\item\label{a_i} For every $i$ the element $a_i$ is not a $p$-th power in $k$ and the ideal $(a_i)$ is prime to the ideals $(a_j)$ and $(b_j)$ for every $j<i$.
\item\label{x_i} Denote by $\sigma_i$ the generator of $\text{Gal}(k(\sqrt[p]{a_i})/k)$ such that $\sigma_i(\sqrt[p]{a_i})=\zeta_p \sqrt[p]{a_i}$. Let $\mathfrak{q}_i$ be a prime of $k$ which splits completely in the extension $k(\sqrt[p]{a_i})/k$ and such that $\mathfrak{q}_i\neq \mathfrak{q}_j$ for every $j<i$. 

Denote by $\beta_1^i,\ldots, \beta_p^i$ the primes of $k(\sqrt[p]{a_i})$ above $\mathfrak{q}_i$ and  by $S_i$ the set of primes of $k(\sqrt[p]{a_i})$ lying above the primes of $k$ dividing $(a_j)$ or $(b_j)$ for $j<i$. Let $x_i\in k(\sqrt[p]{a_i})$ be a solution of the system
\begin{equation}\label{eq_x}
\left\{\begin{array}{cccc} v_{\beta_1^i}(X)=1  & &\\  v_{\beta_j^i}(X)=0  & \forall j\neq 1\\
v_{\beta}(X)=0 & \forall \beta\in S_i.
\end{array}\right.
\end{equation}
and set $b_i:= N_{\sigma_i}(x_i)$, where $N_{\sigma_i}$ is the norm function of the extension $k(\sqrt[p](a_i))/k$.
\end{enumerate}
Finally consider the elements $\omega_1,\ldots,\omega_m$ defined as $$\omega_i:=\frac{b_i^{p-1}}{\prod_{j=i}^{p-1}\sigma_i^j(x_i^j)}.$$
Then:\begin{enumerate}[(a)]\item\label{c1} $\{k(\sqrt[p]{a_i},\sqrt[p]{b_i}, \sqrt[p]{\omega_i}))\}_i$ is a family of linearly disjoint Galois extensions of $k$ and, for every index $i$, the group $H_i:=\text{Gal}(k(\sqrt[p]{a_i},\sqrt[p]{b_i}, \sqrt[p]{\omega_i}))/k)$ is isomorphic to the extraspecial group of order $p^3$ and exponent $p$;\item\label{c2} the extension $k(\sqrt[p]{a_1},\sqrt[p]{b_1},\ldots,\sqrt[p]{a_m}, \sqrt[p]{b_m}, \sqrt[p]{\omega_1\ldots \omega_m})/k$ is a Galois extension with Galois group the extraspecial group of order $p^{2m+1}$ and exponent $p$.
\end{enumerate}
\end{theorem}
\begin{proof}
The first part of property (\ref{c1}) follows from the results of \cite{Micha}, in fact, setting  $$M_i:=k(\sqrt[p]{a_i},\sqrt[p]{b_i},\sqrt[p]{\omega_i}),$$ it is proved in \cite{Micha} that the extensions $M_i/k$ are Galois extensions with Galois group $H_i$ which is an extraspecial group of order $p^3$ and exponent $p$. 

We want to prove that the extensions $M_i$\rq{}s are linearly disjoint and that property (\ref{c2}) holds. We prove the result by induction on $m$.
The theorem is trivially true for $i=1$. 

We suppose the theorem to be true for $i<m$ and we want to prove that the extension $M_m$ is linearly disjoint from $M_i$ for every $i<m$. 
Since the only primes of $k$ which ramify in the extensions $k(\sqrt[p]{a_i})$ and $k(\sqrt[p]{b_i})$ are those above $p$ and those diving the ideals $(a_i)$ and $(b_i)$ respectively, condition $(\ref{a_i})$ ensures that $$k(\sqrt[p]{a_m}) \cap k(\sqrt[p]{a_i})=k(\sqrt[p]{a_m})\cap k(\sqrt[p]{b_i})=k$$ for every $i<m$. 
Since $x_m$ satisfies the system (\ref{eq_x}), we also have $$k(\sqrt[p]{a_m})\cap k(\sqrt[p]{b_m})=k.$$ Moreover we notice that if $\beta$ is a prime of $k$ diving $(a_i)$ or $(b_i)$ for $i<m$, then $\beta$ does not divide $(b_m)$, otherwise $(x_m)$ would be divisible by some elements of $S_m$. Using the previous arguments we get that $k(\sqrt[p]{a_m}, \sqrt[p]{b_m})\cap k(\sqrt[p]{a_i}, \sqrt[p]{b_i})=k$ for every $i<m$.

We notice that an extraspecial group of order $p^3$ has a unique normal subgroup of order $p$, which is given by the center; then the extension $k(\sqrt[p]{a_m}, \sqrt[p]{b_m})$ is the unique $p$-elementary subextension of $M_m$ of order $p^2$. Therefore the extensions $M_1,\ldots, M_m$ are linearly disjoint.

We now want to prove condition (\ref{c2}). We denote by $k_m$ the compositum of the fields $M_1,\ldots, M_m$. Then $\text{Gal}(k_m/k)=H_1\times\ldots\times H_m$ and the field $$M=k(\sqrt[p]{a_1},\sqrt[p]{b_1},\ldots,\sqrt[p]{a_m}, \sqrt[p]{b_m}, \sqrt[p]{\omega_1\ldots \omega_m})$$ is a subfield of $k_m$.

We consider the subgroup $N_m$ of $H_1\times
\ldots \times H_m$ described in Proposition \ref{extra_structure} defined by
$$N_m=\{(z_1^{a_1},\ldots,z_m^{a_m})|a_1+\ldots +a_m\equiv 0 \mod p\}$$ where $Z(H_i)=<z_i>$. It is easy to see that $z_i$ acts in the following way:
\begin{itemize} \item $z_i(\sqrt[p]{\omega_i})=\zeta_p\sqrt[p]{\omega_i}$;
\item $z_i(\sqrt[p]{a_i})=\sqrt[p]{a_i}$; 
\item $z_i(\sqrt[p]{b_i})=\sqrt[p]{b_i}$.
\end{itemize}
The quotient $(H_1\times \ldots \times H_m)/N_m$ is an
extraspecial group of order $p^{2m+1}$ and exponent $p$ and, by Galois correspondence,  the extension $k_m^{N_m}/k$ is an extraspecial extension of order
$p^{2m+1}$ and exponent $p$.
We claim that $M=k_m^{N_m}$. 

The
field $$k(\sqrt[p]{a_1},\sqrt[p]{ b_1},\ldots,
\sqrt[p]{a_m},\sqrt[p]{ b_m})$$ has degree $p^{2m}$ over
$k$ and is contained in $k_m^{N_m}$, since it is fixed
by every element of $N_m$. Thus $k_m^{N_m}$ must be a $p$-extension of this field.

If we take $n=(z_1^{a_1},\ldots,z_m^{a_m})\in N_m$, we have
$$n(\sqrt[p]{ \omega_1\cdot\ldots \cdot
\omega_m})=
{\zeta_p}^{a_1+\ldots +a_m}\sqrt[p]{ \omega_1\cdot\ldots \cdot
\omega_m}=\sqrt[p]{ \omega_1\cdot\ldots \cdot \omega_m}$$ since
$\sum_{i=1}^m a_i\equiv 0 \mod p$ for $n\in N_m$.

Therefore $M\subseteq k_m^{N_m}$ and, since they have the same degree, they are equal.
\end{proof}

\subsection{Extension by the group algebra}\label{realiz_group_algebra}

This section is meant to provide an explicit realization of the field used in the proof of part (\ref{a_noimpli_b}) of Theorem \ref{main_theorem}, as announced in Remark \ref{construction_group_algebra}.

To do this we will resort to considerations from Section \ref{irr_module_extra} on irreducible modules of extraspecial groups over certain finite fields, the explicit realization of extraspecial groups given in the previous section and a method by Serre.

We start by fixing two odd different primes $p$ and $q$. We consider the family of groups $\{E_m\}_m$, where $E_m$ is the extraspecial $p$-group of order $p^{2m+1}$ and exponent $p$ and we denote by $\mathbb{F}_q [E_m]$ the group algebra of $E_m$ over $\mathbb{F}_q$. 

We shall now use the method described in \cite{Serre}, p.18, to prove that if $F$ is a number field containing the $q$-th root of unity, $L_m/F$ is a Galois extension with Galois group $E_m$ and $G_m$ is the semidirect product $G_m=\mathbb{F}_q [E_m]\rtimes E_m$, then the embedding problem for $L_m/F$ and for $$1\rightarrow \mathbb{F}_q[E_m]\rightarrow G_m\rightarrow E_m\rightarrow 1$$ has a solution. This means that there exists a Galois extension $K_m$ of $L_m$ such that:
\begin{enumerate} 
\item $\text{Gal}(K_m/L_m)=\mathbb{F}_q[E_m]$;
\item $K_m/F$ is Galois with Galois group isomorphic to $G_m$.
\end{enumerate}
In order to construct $K_m$, we take a place $v$ of $F$ which splits completely in $L_m$ and we consider a place $w$ of $L_m$ above $v$.
Now we choose an element $x\in L_m$ such that $x$ has $w$-adic valuation $1$, but $w_i$-adic valuation $0$ for every $w_i\neq w$ above $v$.

We denote by $\Delta:=\{\sqrt[q]{\sigma(x)}|\sigma \in E_m\}$ and we set $K_m=L_m(\Delta)$. 
Then the extension $K_m/L_m$ is a Galois extension with Galois group isomorphic to $\mathbb{F}_q[E_m]$. In fact we consider the map $$\lambda:\mathbb{F}_q[E_m]\rightarrow \text{Gal}(K_m/L_m)$$
$$\sum_{\sigma\in E_m}a_{\sigma}\sigma\longmapsto \tau \phantom{rrrrrrr}$$
where $\tau$ is defined as $\tau(\sqrt[q]{\sigma(x)})=\zeta_q^{a_{\sigma}}\sqrt[q]{\sigma(x)}$; hen $\lambda$ is clearly the sought for isomorphism. 

The extension $K_m/F$ is also a Galois extension: in fact for every $\sigma\in \text{Gal}(L_m/F)=E_m$  and every embedding $\tilde{\sigma}:K_m\rightarrow F^{\text{alg}}$ with $\tilde{\sigma}|_{L_m}=\sigma$ we have that, by construction, $\tilde{\sigma}$ is an automorphism of $K_m$.
Finally the Galois group $\text{Gal}(K_m/L_m)$ is isomorphic to the semidirect product $\mathbb{F}_q[E_m]\rtimes E_m$, with $E_m$ acting on $\mathbb{F}_q[E_m]$ by translation. 

Now we set $F=\mathbb{Q}(\zeta_p,\zeta_q)$ and we take $\{L_m\}_m$ to be the family of Galois extensions of $F$ constructed in Section \ref{realiz_extraspecial} with  $\text{Gal}(L_m/F)=E_m$ the extraspecial group of order $p^{2m+1}$ and exponent $p$. For every $m$ we use the above method to solve the embedding problem for $L_m/F$ and for $$1\rightarrow \mathbb{F}_q[E_m]\rightarrow G_m\rightarrow E_m\rightarrow 1$$  and we construct a family of number fields $\{K_m\}_m$ with Galois groups $$\text{Gal}(K_m/\mathbb{Q}(\zeta_p,\zeta_q))=\mathbb{F}_q[E_m]\rtimes E_m.$$ 

In view of Theorem \ref{main_theorem}. (\ref{a_noimpli_b}), the compositum $K$ of the family $K_m$\rq{}s is an infinite algebraic extension of $\mathbb{Q}$ with uniformly bounded local degrees which is not contained in $\mathbb{Q}^{(d)}$ for any positive integer $d$.
\section*{Acknowledgments} I wish to thank Professor Umberto Zannier for suggesting me the problem and guiding me throughout my PhD thesis. I also thank Professor Andrea Lucchini for useful discussions on extraspecial groups and their structure and for his kind help.

\bibliographystyle{amsplain}

\begin{thebibliography}{10}
\bibitem{Zan}\label{Zan} E. Bombieri, U. Zannier, 
\textit{A note on heights in certain infinite extensions of $\mathbb{Q}$}.
Rend. Mat. Acc. Lincei, 12, 2001, pp. 5-14.
\bibitem{Che}\label{Che} S. Checcoli, U. Zannier,
\textit{On fields of algebraic numbers with bounded local degrees}.
Preprint.
\bibitem{Doerk}\label{Doerk} K. Doerk, T. Hawkes, \textit{Finite Solvable
Groups}. De Gruyter, Berlin, 1992.
\bibitem{Kras}\label{Kras} M. Krasner, \textit{Nombre des extensions d'un
degr\'{e} donn\'{e} d'un corps $p$-adique}. Les
Tendances G\'{e}om\'{e}triques en Alg\`{e}bre et Th\'{e}orie des
Nombres, Ed. CNRS, Paris, 1966, pp. 143-169.
\bibitem{Micha}\label{Micha} I. Michailov, \textit{Four non-abelian groups of order $p^4$ as Galois groups}.
 J. Algebra 307, 2007, pp. 287-299.
\bibitem{Serre}\label{Serre} J. P. Serre, H. Darmon, \textit{Topics in Galois
Theory}. Research notes in Mathematics, Jones and Bartlett
Publishers, 1992.
\bibitem{Shaf}\label{Shaf} I. R. Shafarevich, \textit{On p-extensions}. AMS.
Transl., Ser.2 4, 1956, pp. 59-72.
\bibitem{Shaf2}\label{Shaf2} I. R. Shafarevich, \textit{Construction of
fields of algebraic numbers with given solvable Galois group}.
 Izv. Akad. Nauk SSSR Ser. Mat., 18:6, 1954, pp.525-578.
\bibitem{Lee}\label{Lee} M. R. Vaughan-Lee, \textit{The restricted Burnside problem}. Second Ed., Oxford
University Press, 1993.
\bibitem{Widmer}\label{Widmer} M. Widmer \textit{On certain infinite extensions of the rationals with
Northcott property}. Monatsh. Math., 2009.

\end{thebibliography}

\end{document}